\documentclass{amsart}
\usepackage{amssymb}
\usepackage{bbm} % for \mathbbm{k} for denoting a field k = lowercase of \mathbb
\usepackage[nobysame]{amsrefs}
\usepackage{todonotes}
\usepackage{mathpazo}
\usepackage{enumerate}
\usepackage{rotating}
\usepackage{subcaption}
\usepackage{ifthen}
\usepackage{tikz,pgf,pgfmath}
	\usetikzlibrary{shapes,calc,decorations.pathmorphing}
	\pgfdeclarelayer{background}
	\pgfdeclarelayer{middle}
	\pgfsetlayers{background,middle,main}
	\tikzstyle{edge}=[line width=.75pt]
	\tikzstyle{fnode}=[fill=black,draw=black,circle,scale=\s]
	\tikzstyle{pathnode}=[inner sep=.9pt]

\usepackage[colorlinks=true,
	urlcolor=blue!20!black,
	citecolor=green!50!black]{hyperref}
\newtheorem{theorem}{Theorem}[section]
\newtheorem{proposition}[theorem]{Proposition}

\newtheorem{lemma}[theorem]{Lemma}

\theoremstyle{remark}

\newtheorem{example}[theorem]{Example}
\newtheorem{remark}[theorem]{Remark}
\newtheorem{question}[theorem]{Question}
\newcommand{\defn}[1]{{\color{green!50!black}\emph{#1}}}
\newcommand{\defs}{\stackrel{\mathsf{def}}{=}}
\newcommand{\ie}{\text{i.e.}\;}
\newcommand{\Poset}{\mathbf{P}}
\newcommand{\Succ}{\mathsf{Succ}}

\newcommand{\TamariComplex}{\mathcal{T\!C}}
\newcommand{\Asso}{\mathsf{Asso}}
\newcommand{\Dyck}{\mathcal{D}}
\newcommand{\Trees}{\mathcal{T}}
\newcommand{\Cov}{\mathsf{Asso}}
\newcommand{\Tamari}{\mathsf{Tam}}
\newcommand{\ascent}{\mathsf{asc}}
\newcommand{\Ascent}{\mathsf{Asc}}
\newcommand{\valley}{\mathsf{val}}

\newcommand{\return}{\mathsf{ret}}

\newcommand{\horiz}{\mathsf{horiz}}
\newcommand{\hroot}{\mathsf{hroot}}
\renewcommand{\dim}{\mathsf{dim}}
\newcommand{\codim}{\mathsf{codim}}
\renewcommand{\max}{\mathsf{max}}
\renewcommand{\deg}{\mathsf{deg}}

\newcommand{\relevantascent}{\mathsf{rel}}
\newcommand{\Relevantascent}{\mathsf{Rel}}
\newcommand{\corelevant}{\mathsf{corel}}

\newcommand{\interior}{\mathsf{int}}
\newcommand{\Lattice}{\mathbf{L}}
\newcommand{\out}{\mathsf{out}}
\newcommand{\mrk}{\mathsf{mrk}}
\newcommand{\rem}{\mathsf{rem}}
\newcommand{\corem}{\mathsf{corem}}

\newcommand{\Covers}{\mathsf{Edge}}

\title{$F$- and $H$-Triangles for $\nu$-Associahedra}
\author{Cesar Ceballos}
\address{CC: TU Graz, Institut f\"ur Geometrie, Kopernikusgasse 24, 8010 Graz, Austria.}
\email{cesar.ceballos@tugraz.at}
\author{Henri M{\"u}hle}
\address{HM: TU Dresden, Institut f{\"u}r Algebra, Zellescher Weg 12--14, 01069 Dresden, Germany.}
\email{henri.muehle@tu-dresden.de}
\thanks{CC was supported by the Austrian Science Foundation FWF, grant P 33278.  HM has received funding from the European Research Council (Grant Agreement no. 681988, CSP-Infinity).}

\keywords{$\nu$-Tamari lattice, $\nu$-associahedron, $F$-triangle, $H$-triangle}
\subjclass[2010]{05E45, 52B05}

\begin{document}

\allowdisplaybreaks

\begin{abstract}
	For any northeast path $\nu$, we define two bivariate polynomials associated with the $\nu$-associahedron: the $F$- and the $H$-triangle.  We prove combinatorially that we can obtain one from the other by an invertible transformation of variables.  These polynomials generalize the classical $F$- and $H$-triangles of F.~Chapoton in type $A$.  Our proof is completely new and has the advantage of providing a combinatorial explanation of the relation between the $F$- and $H$-triangle.
\end{abstract}

\maketitle

\section{Introduction}
	\label{sec:introduction}
The $\nu$-Tamari lattice is an intriguing object in combinatorics which was originally motivated by enumerative problems in the study of higher trivariate diagonal harmonics. Nowadays, it has applications and connections to other areas, including polytope theory, subword complexes, Hopf algebras, multivariate diagonal harmonics, and parabolic Catalan combinatorics, as well as to the enumeration of various combinatorial objects such as certain lattice walks in the quarter plane, non-crossing tree-like tableaux, and non-separable planar maps, see~\cites{ceballos_hopf_2018,ceballos20the,ceballos19geometry,ceballos20nu,preville17enumeration} and the references therein.   
The $\nu$-Tamari lattice depends on a fixed northeast path $\nu$, and was defined in~\cite{preville17enumeration} as a certain rotation order on the set of \defn{$\nu$-paths}, \ie northeast paths weakly above $\nu$. Alternatively, it can be described in terms of certain binary trees, called \defn{$\nu$-trees}~\cite{ceballos20nu}.

\begin{figure}[h]
	\centering
	\includegraphics[width=.61\textwidth,page=11]{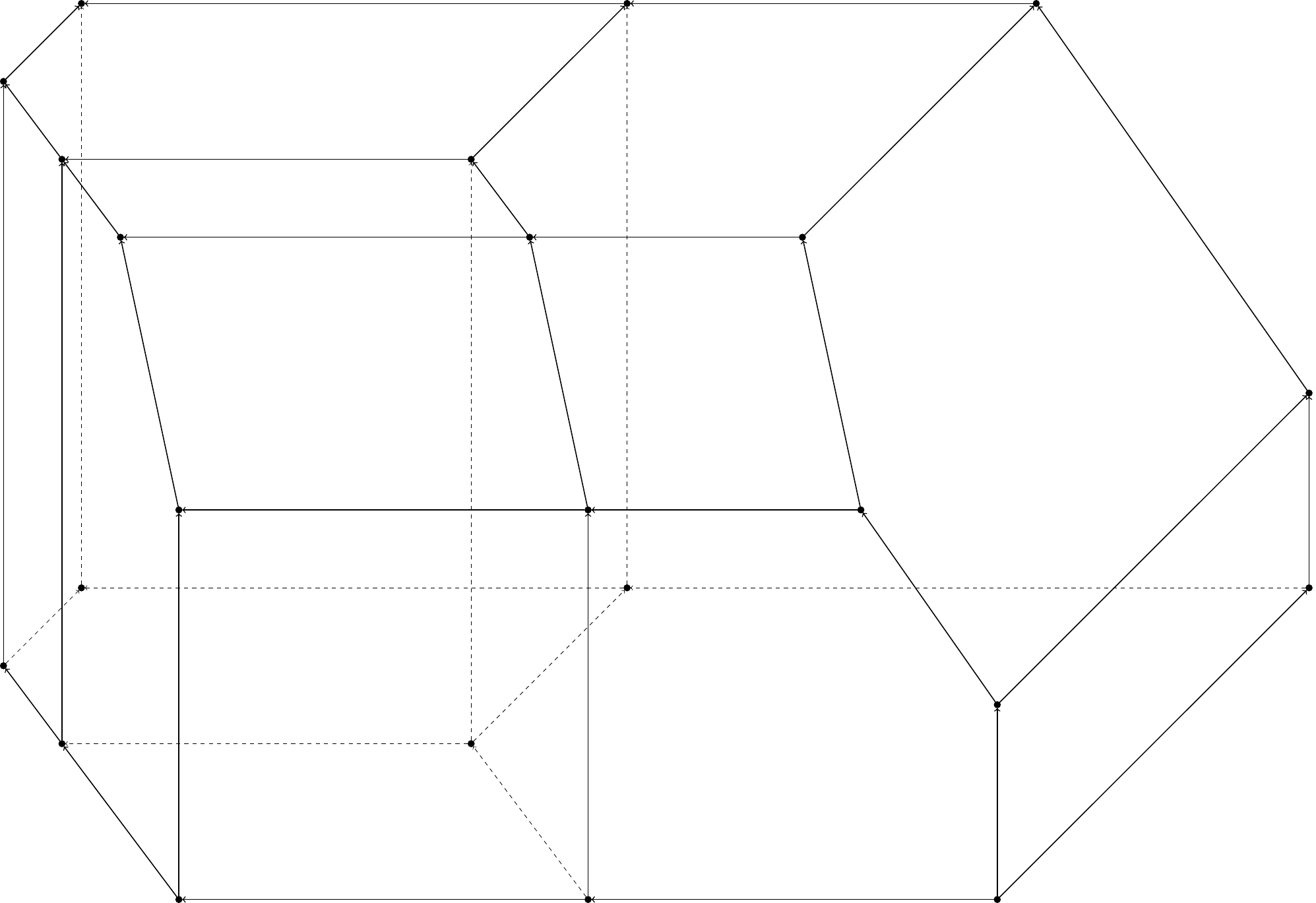}
	\caption{The $\nu$-associahedron for $\nu=ENEENEN$.}
	\label{fig:eneenen_associahedron}
\end{figure}

Motivated by an open problem of F.~Bergeron about the geometry of $m$-Tamari lattices, it was shown in \cite{ceballos19geometry} that the \defn{$\nu$-Tamari lattice} $\Tamari(\nu)$ has a nice underlying geometric structure. They proved that its Hasse diagram can be obtained as the edge graph of a polytopal complex called the \defn{$\nu$-associahedron} $\Asso(\nu)$; see Figure~\ref{fig:eneenen_associahedron}. This complex is dual to a certain triangulation of a particular polytope, which they used to exhibit explicit geometric realizations of the $\nu$-associahedron using techniques from tropical geometry.  The simplicial complex of faces of this triangulation is the \defn{$\nu$-Tamari complex} $\TamariComplex({\nu})$. 

If $\nu=(NE)^{n}$ is the staircase path with $2n$ steps, the corresponding three objects from the previous paragraph are the Tamari lattice~\cite{TamariFestschrift}, the associahedron~\cite{CeballosSantosZiegler2015} and the cluster complex in linear type~$A$ from the theory of cluster algebras~\cites{FominZelevinsky-ClusterAlgebrasII,FZ03}, respectively. The $(NE)^{n}$-paths are better known under the name \defn{Dyck paths}, and we will simply write $\Tamari(n)$, $\Asso(n)$, $\TamariComplex({n})$ rather than $\Tamari\bigl((NE)^{n}\bigr)$, $\Asso\bigl((NE)^{n}\bigr)$, $\TamariComplex\bigl({(NE)^{n}}\bigr)$.  

F.~Chapoton has observed a remarkable enumerative connection between the cluster complex $\TamariComplex({n})$ and the set of Dyck paths~\cites{chapoton06sur}.  More precisely, he defined the following two polynomials:
\begin{itemize}
	\item the \defn{$F$-triangle} $F_{n}(x,y)$ is the bivariate generating function of the faces of $\TamariComplex({n})$, where the variable $x$ accounts for so-called positive roots per face and $y$ accounts for so-called negative simple roots;
	\item the \defn{$H$-triangle} $H_{n}(x,y)$ is the bivariate generating function of Dyck paths, where the variable $x$ accounts for the valleys per path and $y$ accounts for the returns.\footnote{F.~Chapoton introduced the $H$-triangle in the context of his study of the cohomology of the toric variety associated with a fan arising from finite type cluster algebras~\cite{chapoton_personalCommunication_2021}.}
\end{itemize}
He then conjectured that these polynomials are related by the following invertible transformation:
\begin{equation}\label{eq:hf_transformation}
	F_{n}(x,y) = x^{n-1}H_{n}\left(\frac{x+1}{x},\frac{y+1}{x+1}\right).
\end{equation}
This conjecture was generalized for Fu{\ss}--Catalan families by D.~Armstrong~\cite{armstrong09generalized}, and was proven in this general setting by M.~Thiel~\cite{thiel14on}*{Theorem~2}.  Thiel's proof makes clever use of a combinatorial bijection on so-called $k$-generalized nonnesting partitions which leads to a differential equation involving the $H$-triangle.  Using a differential equation by C.~Krattenthaler involving the $F$-triangle, he then proves~\eqref{eq:hf_transformation} by induction.

Unfortunately, the combinatorial nature of the relation between the $F$- and the $H$-triangle is obscured in Thiel's proof.  The main result of the present article is a combinatorial proof of a generalization of \eqref{eq:hf_transformation} to $\nu$-paths and the $\nu$-associahedron.  

Given any northeast path $\nu$, we denote by $\deg(\nu)$ the maximal number of valleys that a northeast path weakly above $\nu$ can have. {In other words, $\deg(\nu)$ describes the size of the largest staircase shape that fits above $\nu$ in the rectangle enclosing~$\nu$.}  The $H$-triangle associated with $\nu$ is simply the bivariate generating function of $\nu$-paths, denoted by $H_{\nu}(x,y)$, where the variable $x$ accounts for valleys and the variable $y$ accounts for returns.  The $F$-triangle is the bivariate generating function $F_{\nu}(x,y)$ of the faces of $\Asso(\nu)$, where $x$ and $y$ account for a new pair of statistics that we introduce in this paper. Our main result shows that these polynomials satisfy \eqref{eq:hf_transformation}.

\begin{theorem}\label{thm:fh_correspondence}
	For every northeast path $\nu$, the following holds:
	\begin{equation}\label{eq:h_to_f}
		F_{\nu}(x,y) = x^{\deg(\nu)}H_{\nu}\left(\frac{x+1}{x},\frac{y+1}{x+1}\right).
	\end{equation}
	Equivalently,
	\begin{equation}\label{eq:f_to_h}
		H_{\nu}(x,y) = (x-1)^{\deg(\nu)}F_{\nu}\left(\frac{1}{x-1},\frac{x(y-1)+1}{x-1}\right).
	\end{equation}
\end{theorem}

Our proof of Theorem~\ref{thm:fh_correspondence} is completely combinatorial. It relies on the geometry of the $\nu$-associahedron and exploits a bijection of \cite{ceballos20nu} which sends $\nu$-paths to $\nu$-trees.  

If $\nu=N^{a_{1}}E^{a_{1}}N^{a_{2}}E^{a_{2}}\cdots N^{a_{r}}E^{a_{r}}$ {for positive integers $a_{1},a_{2},\ldots,a_{r}$}, then Theorem~\ref{thm:fh_correspondence} sheds quite some light on the constructions from \cite{muehle21noncrossing}*{Section~5} and \cite{muehle19ballot}*{Section~5}.  If moreover $a_{2}=a_{3}=\cdots=a_{r}=1$, then our $F$-triangle combinatorially realizes the {case $m=1$} of the $F$-triangle computed {abstractly} in \cite{krattenthaler19the}*{Theorem~4.3}.
We wish to remark that analogues of $F$- and $H$-triangles arising in different (geometric) contexts but satisfying \eqref{eq:hf_transformation}, too, were for instance considered in \cites{garver20chapoton,muehle20hochschild}.

This article is organized as follows.  In Section~\ref{sec:basics}, we recall the basic definitions surrounding the $\nu$-Tamari lattice, such as $\nu$-paths, $\nu$-trees, rotation, the $\nu$-Tamari lattice, the $\nu$-associahedron and the $\nu$-Tamari complex.  In Section~\ref{sec:nu_triangles}, we introduce two bivariate polynomials arising naturally in the context of $\nu$-Tamari lattices and $\nu$-associahedra and we realize them in terms of $\nu$-trees and certain statistics.  We prove our main result (Theorem~\ref{thm:fh_correspondence}) in Section~\ref{sec:main_proof}, and present a generalization to arbitrary posets as Theorem~\ref{thm:fh_poset_correspondence} in Section~\ref{sec:fh_correspondence_posets}.  We conclude this article with a reciprocity result for the $\nu$-Tamari complex in Section~\ref{sec:nu_tamari_reciprocity}, which is the foundation for a generalization that we present in an upcoming note.

\section{Basics}
	\label{sec:basics}
\subsection{Northeast paths}
A \defn{northeast path} is a lattice path in $\mathbb{N}^{2}$ starting at the origin, and consisting of finitely many steps of the form $(0,1)$ (\defn{north steps}) and $(1,0)$ (\defn{east steps}).  We write such a path as a word over the alphabet $\{N,E\}$, where each $N$ represents a north step and each $E$ an east step.
Throughout this paper, we let $\nu$ denote (a fixed) such northeast path.  Let $F_{\nu}$ denote the Ferrers diagram that lies weakly above $\nu$ in the smallest rectangle containing~$\nu$.  Let~$A_{\nu}$ denote the set of lattice points inside $F_{\nu}$.  See Figure~\ref{fig:nu_ferrers} for an illustration.

\begin{figure}
	\centering
	\begin{subfigure}[t]{.45\textwidth}
		\centering
		\includegraphics[scale=1,page=1]{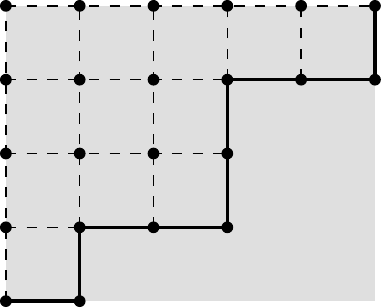}
		\caption{A northeast path $\nu=ENEENNEEN$, its bounding rectangle (shaded in gray), the associated Ferrers diagram~$F_{\nu}$ (indicated by the dashed lines), the associated set of lattice points~$A_{\nu}$ (indicated by the black dots).}
		\label{fig:nu_ferrers}
	\end{subfigure}
	\hspace*{1cm}
	\begin{subfigure}[t]{.45\textwidth}
		\centering
		\includegraphics[page=19,scale=1]{3d_associahedron.pdf}
		\caption{Illustrating the horizontal distance of lattice points on a $\nu$-path.}
		\label{fig:horizontal_distance}
	\end{subfigure}
	\caption{Illustrating some basic definitions for $\nu$-paths.}
\end{figure}

\subsection{The $\nu$-Tamari lattice}
	\label{sec_TamariLattice_paths}
Let us denote by $\Dyck_{\nu}$ the set of all \defn{$\nu$-paths}, \ie northeast paths that live entirely inside $F_{\nu}$ sharing start and end points with~$\nu$ and lie weakly above $\nu$.  For $\mu\in\Dyck_{\nu}$, a \defn{valley} is a point $p\in A_{\nu}$ which lies on $\mu$ and is preceded by an east step and followed by a north step.  We denote by $\valley(\mu)$ the number of valleys of $\mu$.  A valley $p$ of $\mu$ is a \defn{return}, if $p$ is also a valley of $\nu$.  We denote by $\return(\mu)$ the number of returns of $\mu$.
The \defn{degree} of $\nu$ is defined as the maximum number of valleys that a $\nu$-path can have: 
\begin{displaymath}
	\deg(\nu) \defs \max\bigl\{\valley(\mu)\mid\mu\in\Dyck_{\nu}\bigr\}.
\end{displaymath}

If $p\in A_{\nu}$, then we denote by $\horiz_{\nu}(p)$ the \defn{horizontal distance} of $p$ to the right boundary of $F_{\nu}$, \ie the maximal number of east steps that we can append to $p$ without leaving $F_{\nu}$.  In other words, if $p=(i,j)$, then we look for the rightmost point in row $j$ that lies in $F_{\nu}$; say that this point is $(k,j)$.  Then
\begin{displaymath}
	\horiz_{\nu}(p)\defs k-i.
\end{displaymath}
Figure~\ref{fig:horizontal_distance} shows an element of $\Dyck_{ENEENNEEN}$, where each lattice point is labeled by its horizontal distance.

If $p$ is a valley of $\mu\in\Dyck_{\nu}$, then let $p'$ denote the first lattice point on $\mu$ after $p$ with $\horiz_{\nu}(p')=\horiz_{\nu}(p)$.  Let $\mu[p,p']$ denote the subpath of $\mu$ which lies between $p$ and $p'$.  The \defn{rotation} of $\mu$ by $p$ is the unique northeast path which arises from $\mu$ by swapping the east step before $p$ with $\mu[p,p']$.  If $\mu'$ is the path arising from $\mu$ in this manner, then we write $\mu\lessdot_{\nu}\mu'$.  It is quickly verified that $\lessdot_{\nu}$ is an acyclic binary relation on $\Dyck_{\nu}$, and we denote its reflexive and transitive closure by $\leq_{\nu}$.  See Figure~\ref{fig:nu_rotation} for an illustration.

\begin{figure}
	\centering
	\includegraphics[scale=1,page=2]{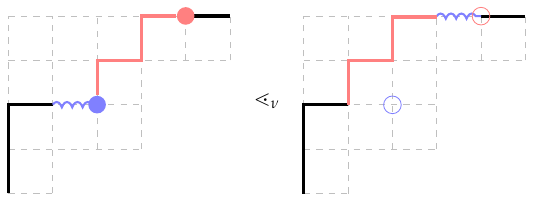}
	\caption{An example of a rotation of a $\nu$-path by the valley marked in blue.}
	\label{fig:nu_rotation}
\end{figure}

The partially ordered set $\Tamari(\nu)\defs(\Dyck_{\nu},\leq_{\nu})$ is a lattice; the \defn{$\nu$-Tamari lattice}; see~\cite{preville17enumeration}*{Theorem~1.1}.  Figure~\ref{fig:eenen_tamari_paths_fhlabel} shows the $\nu$-Tamari lattice for the path $\nu=EENEN$, which has degree $2$. 

\begin{figure}
	\centering
	\includegraphics[scale=1,page=3]{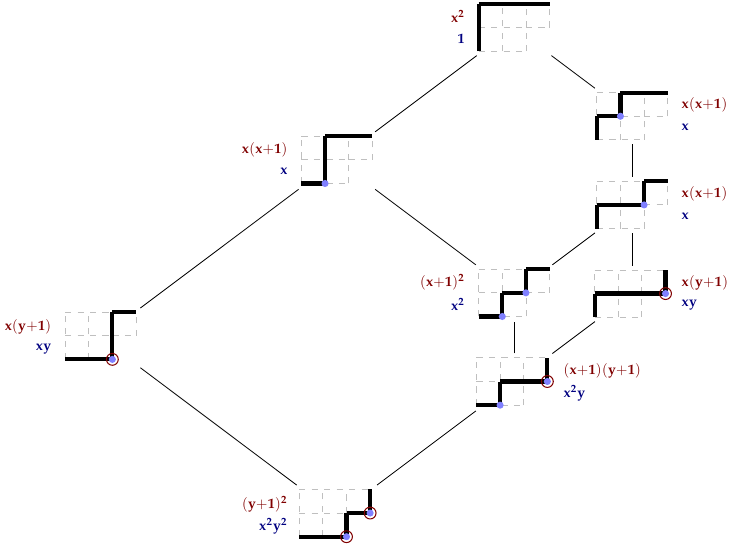}
	\caption{The $\nu$-Tamari lattice labeled by $\nu$-paths for $\nu=EENEN$.  Each path is additionally labeled by the term it contributes to $F_{\nu}(x,y)$ (top expression in red) and to $H_{\nu}(x,y)$ (bottom expression in blue).}
	\label{fig:eenen_tamari_paths_fhlabel}
\end{figure}

\subsection{The $\nu$-Tamari lattice via trees}
	\label{sec:nu_tamari_trees}
As shown in~\cite{ceballos20nu}, we can alternatively define the $\nu$-Tamari lattice in terms of a special family of trees. 

We say that two points $p,q\in A_{\nu}$ are \defn{$\nu$-incompatible} if $p$ is strictly southwest or strictly northeast of $q$ and the smallest rectangle containing $p$ and $q$ lies entirely in~$F_{\nu}$.  Otherwise, $p$ and $q$ are \defn{$\nu$-compatible}; we write $p\sim_{\nu}q$ in this case, and drop the reference to the path if no confusion may arise. 
A \defn{$\nu$-tree} is a maximal collection of pairwise $\nu$-compatible elements of $A_{\nu}$. 
We denote by $\Trees_\nu$ the set of all $\nu$-trees. 

If $T$ is a $\nu$-tree, we can connect two distinct elements $p,q\in T$ if $p$ and $q$ either lie in the same row or in the same column, and there is no element of $T$ on the line segment connecting $p$ and $q$. In particular, this allows us to visualize $\nu$-trees as classical rooted binary trees \cite{ceballos20nu}*{Lemma~2.4}.  An example is shown in Figure~\ref{fig:nu_tree}.

\begin{figure}[h]
	\begin{subfigure}[t]{.45\textwidth}
		\centering
		\includegraphics[scale=1,page=4]{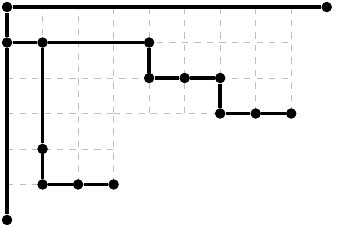}
		\caption{A $\nu$-tree for $\nu=NE^{3}N^{2}E^{5}N^{3}E$.}
		\label{fig:nu_tree_ne3n2e5n3e}
	\end{subfigure}
	\hspace*{1cm}
	\begin{subfigure}[t]{.45\textwidth}
		\centering
		\includegraphics[scale=1,page=5]{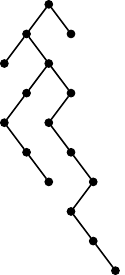}
		\caption{The representation of the tree from Figure~\ref{fig:nu_tree_ne3n2e5n3e} as an ordinary binary tree.}
		\label{fig:nu_tree_ne3n2e5n3e_binary}
	\end{subfigure}
	\caption{A $\nu$-tree and its visualization as a classical rooted binary tree.}
	\label{fig:nu_tree}
\end{figure}

Let $T\in\Trees_{\nu}$ and let $p,q\in T$ be two elements which do not lie in the same row or same column.  Let $p\square r$ denote the smallest rectangle containing $p$ and $r$.  We write $p\llcorner r$ (resp. $p\urcorner r$) for the lower left corner (resp. upper right corner) of $p\square r$.  

An element $q\in T$ is an \defn{ascent} of $T$ if $q=p\llcorner r$ for some elements $p,r\in T$. In such a case, we choose $p,r$ canonically so that no other elements besides $q,p,r$ lie in $p\square r$.  We denote the set of ascents of $T$ by $\Ascent(T)$, and write $\ascent(T)\defs\bigl\lvert\Ascent(T)\bigr\rvert$.  

The \defn{rotation} of $T$ by the ascent $q$ is $T'=\bigl(T\setminus\{q\})\cup\{q'\}$, where $q'=p\urcorner r$.  Figure~\ref{fig_treerotation} illustrates this rotation operation.  As proven in~\cite{ceballos20nu}*{Lemma~2.10}, the rotation of a $\nu$-tree is also a $\nu$-tree.  By abuse of notation, we write $T\lessdot_{\nu} T'$ if $T'$ is a rotation of $T$, and denote by $\leq_{\nu}$ the reflexive and transitive closure of $\lessdot_{\nu}$.  The partial order $(\Trees_{\nu},\leq_{\nu})$ is a lattice, {which is} isomorphic to the $\nu$-Tamari lattice~\cite{ceballos20nu}*{Theorem~3.3}.

\begin{figure}
	\centering
	\includegraphics[scale=1,page=6]{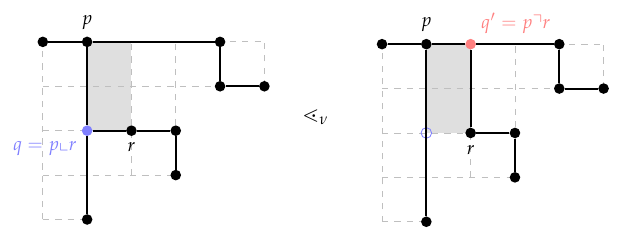}
	\caption{The rotation operation of a $\nu$-tree by the ascent node $q$.  The rectangle $p\square r$ is highlighted.  See also Figure~\ref{fig:nu_rotation}.}
	\label{fig_treerotation}
\end{figure}     

\subsection{The right flushing bijection}\label{sec_rightflushingbijection}
The isomorphism between the $\nu$-Tamari lattice and the rotation lattice of $\nu$-trees is given by a simple bijection between the set of $\nu$-paths and the set of $\nu$-trees which we now recall. Given a $\nu$-path $\mu$, let $a_i$ be the number of lattice points on $\mu$ at height $i$, for $i\geq 0$.  There exists exactly one $\nu$-tree $T$ containing $a_i$ nodes at height $i$ for each $i\geq 0$. Vice-versa, given a $\nu$-tree with ``height sequence'' $a_0,a_1,a_2,\dots$, there is a unique $\nu$-path with the same height sequence.
We denote by $\Phi\colon\Dyck_\nu \to \Trees_\nu$ the map that sends $\mu$ to~$T$. This map is a bijection between the set of $\nu$-paths and the set of $\nu$-trees.   Moreover, it is an isomorphism between the $\nu$-Tamari lattice and the rotation lattice of $\nu$-trees~\cite{ceballos20nu}*{Proposition~16}. The map $\Phi$ is called the \defn{right flushing bijection}~\cite{ceballos20nu}, and is illustrated in Figure~\ref{fig:right_flushing}.

The reason why this is called ``right flushing'' is because it can be described as follows. 
Let $\mu$ be a $\nu$-path with height sequence $a_0,a_1,a_2,\dots$. We build the $\nu$-tree $T=\Phi(\mu)$ with the same height sequence by recursively adding $a_i$ nodes at height $i$ from bottom to top, from right to left, avoiding forbidden positions. The forbidden positions are those above a node that is not the left most node in a row (these come from the initial points of the east steps in the path $\mu$). In Figure~\ref{fig:right_flushing}, the forbidden positions are the ones that belong to the wiggly lines.  Note that the order of the nodes per row is reversed.

\begin{figure}
	\centering
	\includegraphics[scale=.85,page=7]{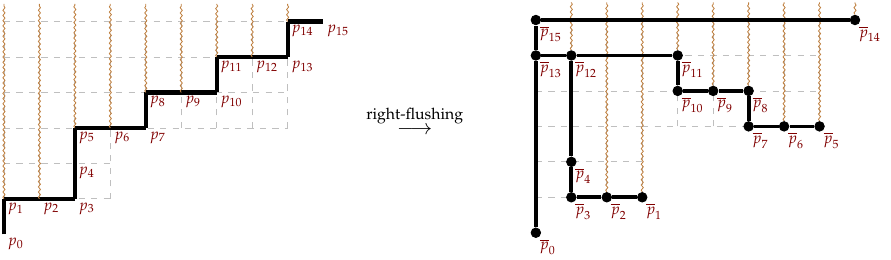}
	\caption{Illustrating the bijection from $\nu$-paths to $\nu$-trees.}
	\label{fig:right_flushing}
\end{figure}

\subsection{The $\nu$-Tamari complex and the $\nu$-associahedron}
Generalizing the $\nu$-trees mentioned above, we define a \defn{$\nu$-face} as a collection of pairwise $\nu$-compatible elements of $A_{\nu}$ (not necessarily maximal as in the case of $\nu$-trees).  The collection of $\nu$-faces forms a simplicial complex, which we call the \defn{$\nu$-Tamari complex} and denote by $\TamariComplex({\nu})$. 
This complex was originally defined using a different language in~\cite{ceballos19geometry}, and we use the terminology introduced in~\cite{ceballos20nu}.  

The $\nu$-Tamari complex is the simplicial complex of faces of a triangulation of a polytope studied in~\cite{ceballos19geometry}. The dual of this triangulation is a polytopal complex called the {$\nu$-associahedron} $\Asso(\nu)$, whose faces are in correspondence (via duality) with the interior faces of the triangulation. Such interior faces were classified in~\cite{ceballos19geometry} as \defn{covering $\nu$-faces}, which are defined as those $\nu$-faces containing the top-left corner of~$A_{\nu}$ and at least one point in each row and column in $F_{\nu}$~\cites{ceballos19geometry,ceballos20nu}. 
Covering $\nu$-faces are also called $\nu$-Schr{\"o}der trees in~\cite{bell20schroder}. 
We keep the name ``covering $\nu$-faces" because it appeared first in~\cite{ceballos19sweakorder}*{Definition~5.3}, following the conventions in~\cite{ceballos20nu}. This name is based on the terminology ``covering $(I,J)$-forests" used on the original definition of the $\nu$-associahedron in~\cite{ceballos19geometry}.

The \defn{$\nu$-associahedron} is defined as the polytopal complex whose cells are covering $\nu$-faces ordered by reversed inclusion.
\begin{displaymath}
	\Asso(\nu) \defs \{C\mid \text{$C$ is a covering $\nu$-face}\}.
\end{displaymath}
If $\nu$ is a northeast path from $(0,0)$ to~$(m,n)$, the \defn{dimension} of a covering~$\nu$-face $C$ is:
\begin{displaymath}
	\dim(C) \defs m+n+1 - \lvert C\rvert.
\end{displaymath}
In particular, one can check that every $\nu$-tree has $m+n+1$ elements. So, the $\nu$-trees correspond to the zero-dimensional faces (vertices) of the $\nu$-associahedron. Every time we remove a node (when possible), we increase the dimension of the resulting face by one. 
One can also see, for instance from Lemma~\ref{lem_asso_characterization_faces} and Lemma~\ref{lem_changingstatistics}~\eqref{lem_changingstatistics_one} further below, that the maximal dimension of a face in $\Asso(\nu)$ is equal to the maximal number of valleys that a $\nu$-path can have. Therefore,
\begin{displaymath}
	\dim\bigl(\Asso(\nu)\bigr) = \deg(\nu).
\end{displaymath}
An example of the $\nu$-associahedron for $\nu=EENEN$ is illustrated in Figure~\ref{fig:eenen_face}. The faces of this figure are labeled by covering $\nu$-faces and the vertices by $\nu$-trees. Its edge graph coincides with the Hasse diagram of the $\nu$-Tamari lattice in Figure~\ref{fig:eenen_tamari_paths_fhlabel}.
The advantage of working with the $\nu$-associahedron is that it captures the full geometric information behind the $\nu$-Tamari lattice. 

\begin{figure}
	\centering
	\includegraphics[scale=1,page=8]{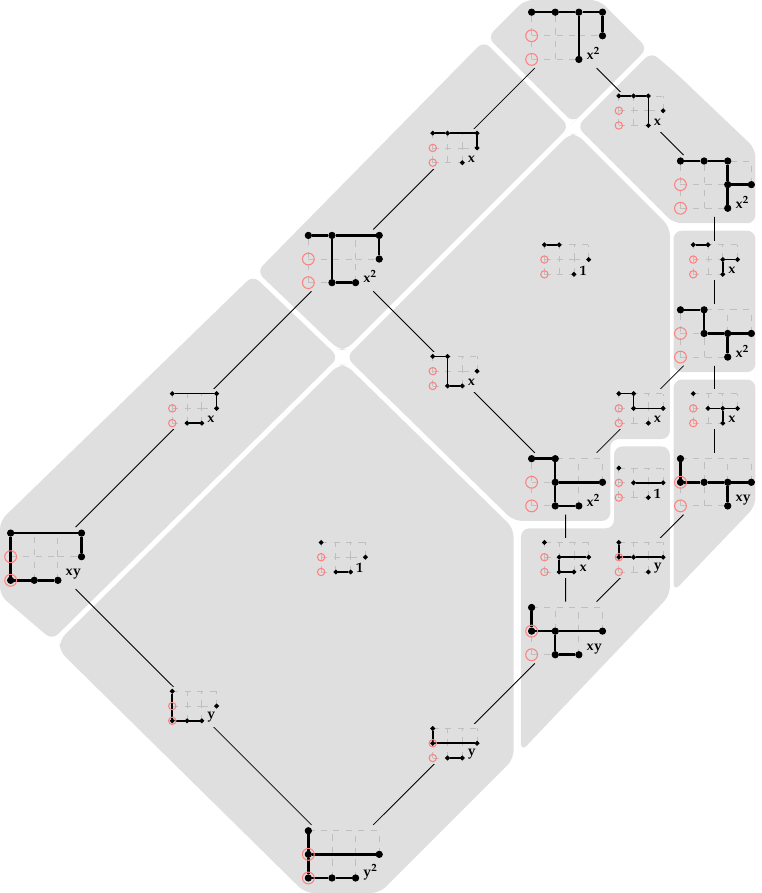}
	\caption{The $\nu$-associahedron $\Asso(\nu)$ for $\nu=EENEN$, whose faces are labeled by covering $\nu$-faces. Each face is additionally labeled by the term it contributes to $F_{\nu}(x,y)$.}
	\label{fig:eenen_face}
\end{figure}

\subsection{The $\nu$-associahedron via $\nu$-Schr{\"o}der paths}\label{sec_schroeder}
The faces of the $\nu$-associahedron can also be labeled in terms of another nice family of combinatorial objects called $\nu$-Schr{\"o}der paths~\cite{bell20schroder}.
A \defn{$\nu$-Schr{\"o}der path} is a lattice path consisting of north steps~$N$, east steps $E$, and diagonal steps $D=(1,1)$, that shares the start and end points with~$\nu$ and lies weakly above $\nu$. It was shown in~\cite{bell20schroder}*{Section~3.1} that the set of $\nu$-Schr{\"o}der paths is in bijection with the set of covering $\nu$-faces (called \emph{$\nu$-Schr{\"o}der trees} in that paper). The bijection is essentially the same as the right flushing bijection from Section~\ref{sec_rightflushingbijection}, with the small difference that the forbidden positions are those that are above nodes corresponding to the initial points of the east and diagonal steps of the $\nu$-Schr{\"o}der path.
This bijection is illustrated in Figure~\ref{fig_right_flushing_SchroederPaths}. 

\begin{figure}
	\centering
	\includegraphics[scale=.85,page=26]{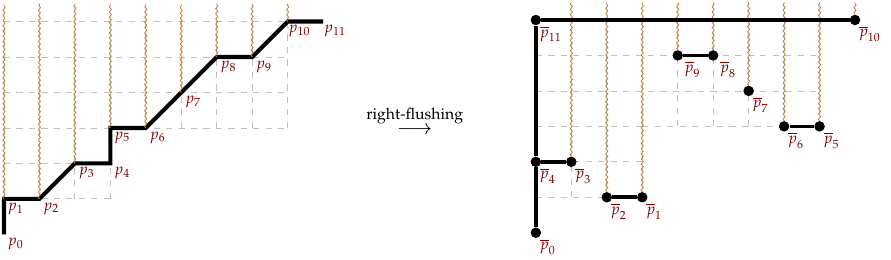}
	\caption{Illustrating the bijection from $\nu$-Schr{\"o}der paths to covering $\nu$-faces from~\cite{bell20schroder}.}
	\label{fig_right_flushing_SchroederPaths}
\end{figure}

The faces of the $\nu$-associahedron $\Asso(\nu)$ are therefore in correspondence with $\nu$-Schr{\"o}der paths. The dimension of a face associated with a $\nu$-Schr{\"o}der path is equal to the number of diagonal steps. The face poset of $\Asso(\nu)$ can also be described in terms of a poset on $\nu$-Schr{\"o}der paths but the definition is a bit more involved, see~\cite{bell20schroder}*{Definition~3.12 and Theorem~4.7}.  
We have redrawn the $\nu$-associahedron from Figure~\ref{fig:eenen_face} in terms of $\nu$-Schr{\"o}der paths in Figure~\ref{fig:eenen_face_schroeder}.

\begin{figure}
	\centering
	\includegraphics[scale=1,page=25]{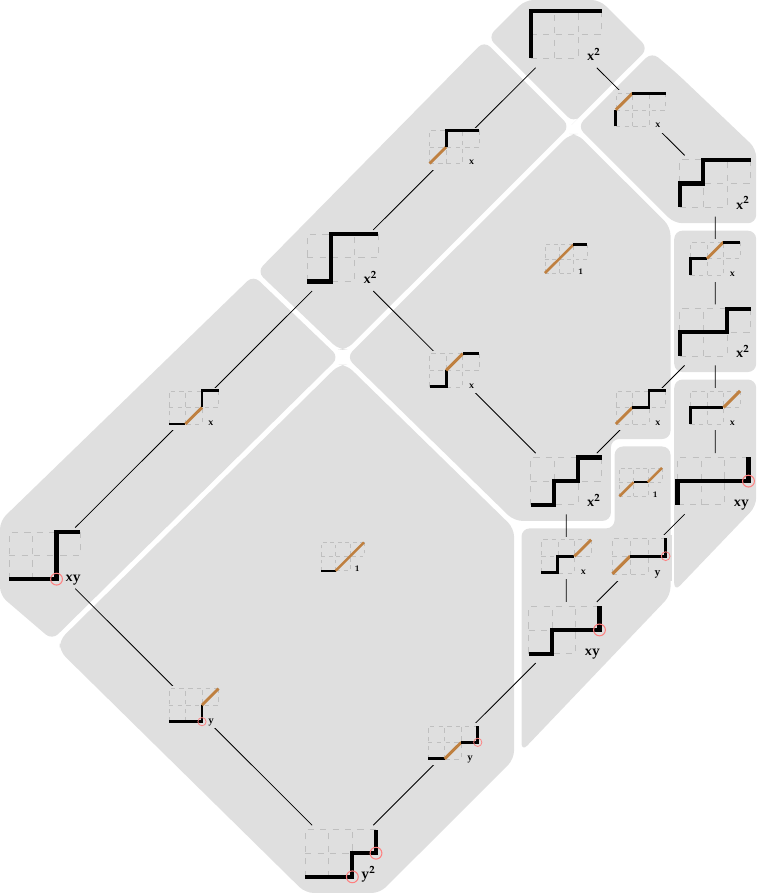}
	\caption{The $\nu$-associahedron $\Asso(\nu)$ for $\nu=EENEN$, whose faces are labeled by $\nu$-Schr{\"o}der trees. Each face is additionally labeled by the term it contributes to $F_{\nu}(x,y)$.}
	\label{fig:eenen_face_schroeder}
\end{figure}

\section{The $F$- and the $H$-triangle associated with $\nu$}
	\label{sec:nu_triangles}
Let $C\in \Cov(\nu)$ be a covering $\nu$-face.
We say that $p\in C$ is \defn{relevant} if:
\begin{itemize}
	\item it is in the first column, 
	\item there is another point $q\neq p$ in $C$ that is in the same row, and
	\item its row contains a valley of $\nu$.%\cesar{removed the non-root condition which follows from the last}
\end{itemize}
We denote by $\Relevantascent(C)$ the set of relevant nodes in $C$, and we let 
\begin{equation}\label{eq_rel}
	\relevantascent(C) \defs\bigl\lvert\Relevantascent(C)\bigr\vert.
\end{equation}
We also define the statistic
\begin{align}\label{eq_corel}
	\corelevant(C) & \defs \deg(\nu)-\dim(C)-\relevantascent(C) \\
	 & = \codim(C)-\relevantascent(C),
\end{align}
where $\codim(C) \defs \deg(\nu)-\dim(C)= \dim\bigl(\Asso(\nu)\bigr)-\dim(C)$ denotes the codimension of the face $C$ in the~$\nu$-associahedron $\Asso(\nu)$.

The \defn{$F$-triangle} of $\Asso(\nu)$ is a generating function of the faces of $\Asso(\nu)$ defined by:
\begin{equation}\label{eq:f_triangle}
	F_{\nu}(x,y) \defs \sum_{C\in \Cov(\nu)} x^{\corelevant(C)} y^{\relevantascent(C)}.
\end{equation}
Note that the degree of the term associated with $C$ is $\corelevant(C)+\relevantascent(C)=\codim(C)$.

In Figure~\ref{fig:eenen_face}, the positions of the relevant nodes are circled in red in order to easily visualize the value of the statistic $\relevantascent(C)$ on each face. The degree is $\deg(\nu)= 2 = \dim(\Asso(\nu))$. In addition, each face $C$ is labeled by the term it contributes to the $F$-triangle, whose degree is $\codim(C)$ (vertices have degree 2, edges degree 1, and 2-faces degree 0). Adding up, we obtain
\begin{displaymath}
	F_{EENEN}(x,y) = 5x^{2} + 3xy + y^{2} + 8x + 3y + 3.
\end{displaymath}

\begin{remark}
The $F$-triangle defined in Equation~\eqref{eq:f_triangle} is a very natural generalization of Chapoton's original definition of the $F$-triangle for cluster complexes of type~$A$. Indeed, as we will explain below, clusters of type $A_n$ can be identified with covering $\nu$-faces for the stair case path $\nu=(EN)^n$, and this identification transforms the $\relevantascent$ and $\corelevant$ statistics of a covering $\nu$-face to the statistics counting the negative simple roots and positive roots of the corresponding cluster, respectively. These two statistics are the statistics used in Chapoton's original definition of the $F$-triangle for cluster complexes.

The bijection between covering $\nu$-faces for the staircase path $\nu=(EN)^n$ and clusters of type $A_n$ works as follows.
First, observe that every covering $\nu$-face contains the top left corner of the Ferrers diagram above $\nu$, as well as all the valleys of $\nu$.
Label all other integer points in the Ferrers diagram with almost positive roots as in Figure~\ref{fig:staircase_roots}: the negative simple roots are placed on the first column, and the positive roots are placed forming the triangular missing part as shown. The cluster associated to a covering $\nu$-face is just the set of its labels; the fact that this is a bijection follows \ie from~\cite{ceballos001}*{Theorem~2.2} in combination with~\cite{ceballos20nu}*{Section~5}. 

Under this correspondence, $\relevantascent(C)$ coincides with the number of negative simple roots of the associated cluster, because it counts the number of elements of $C$ in the first column except for the top left corner. 
On the other hand, $\codim(C)$ is equal to the number of elements of $C$ which are not the top left corner or a valley of $\nu$. Since $\corelevant(C)=\codim(C)-\relevantascent(C)$, this statistic is counting the number of positive roots of the corresponding cluster.  
\end{remark}

\begin{remark}
Following Section~\ref{sec_schroeder}, one can alternatively define the $F$-triangle as a generating function on $\nu$-Schr{\"o}der paths $\pi$:
\begin{equation}\label{eq_f_triangle_schroeder} 
	F_{\nu}(x,y) = \sum_{\pi} x^{\codim(\pi)-\return(\pi)} y^{\return(\pi)}.
\end{equation}
The statistic $\return(\pi)$ is the number of \defn{returns} of $\pi$ (which are defined as valleys of~$\pi$ which are also valleys of $\nu$).
The reason that this is the desired statistic is that, under the right flushing bijection, such returns are precisely the lattice points of~$\pi$ that are mapped to the relevant nodes of the corresponding covering $\nu$-face~$C$; in other words, $\return(\pi)=\relevantascent(C)$.
The term $\codim(\pi)=\codim(C)$ is the codimension of the face $C\in \Asso(\nu)$; since the dimension of this face is equal to the number of diagonal steps of $\pi$, we have that $\codim(\pi)$ is equal to $\deg(\nu)$ minus the number of diagonal steps in $\pi$.
\end{remark}

The \defn{$H$-triangle} of $\nu$ is the generating function of the elements of $\Dyck_{\nu}$ in terms of the number of valleys and returns:
\begin{equation}\label{eq:h_triangle}
	H_{\nu}(x,y) \defs \sum_{\mu\in\Dyck_{\nu}}x^{\valley(\mu)}y^{\return(\mu)}.
\end{equation}
In Figure~\ref{fig:eenen_tamari_paths_fhlabel}, we have marked the valleys in each path by a blue dot, and we have circled the returns in red.  Additionally, we have noted the term each path contributes to $H_{EENEN}(x,y)$ in blue (bottom expression).  Adding up, we obtain
\begin{displaymath}
	H_{EENEN}(x,y) = x^{2}y^{2} + x^{2}y + x^{2} + 2xy + 3x + 1.
\end{displaymath}

\begin{figure}
	\centering
	\includegraphics[page=23,scale=1]{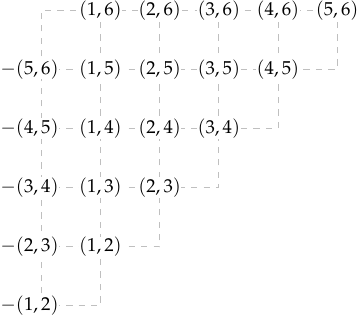}
	\caption{Labeling lattice points in the staircase shape by almost positive roots.}
	\label{fig:staircase_roots}
\end{figure}

\begin{remark}
	It is worth mentioning that the two polynomials $F_{\nu}(x,y)$ and $H_{\nu}(x,y)$ associated to a path $\nu$ remain unchanged after adding north steps at the beginning of the path and east steps at the end. In other words, if $\nu'=N^{a}\nu E^{b}$ then $F_{\nu'}(x,y)=F_{\nu}(x,y)$ and $H_{\nu'}(x,y)=H_{\nu}(x,y)$. Therefore, we may assume without loss of generality that $\nu$ starts with an east step and ends with a north step.  This assumption, however, has no impact on our proofs.
	
	The fact that $H_{\nu'}(x,y)=H_{\nu}(x,y)$ is straightforward: for any $\mu\in\Dyck_{\nu}$, we have $N^{a}\mu E^{b}\in\Dyck_{\nu'}$, and this is clearly a bijective correspondence.  Moreover, since the $H$-triangle enumerates paths with respect to the number of valleys and returns, it is clear that $H_{\nu'}(x,y)=H_{\nu}(x,y)$, because no valley can occur in the first column or the last row.  
	
	The fact that $F_{\nu'}(x,y)=F_{\nu}(x,y)$ can be explained as follows. Geometrically, any lattice point in $A_{\nu'}$ which occurs in the prefix $N^{a}$ or in the suffix $E^{b}$ is compatible with any other lattice point in $A_{\nu'}$. This implies that $\TamariComplex(\nu')$ is isomorphic to the join of $\TamariComplex(\nu)$ with $a+b$ single points $p_1,\dots,p_{a+b}$.  On the other hand, since every covering $\nu$-face is required to have at least one point in each row and column, these $a+b$ points belong to every covering $\nu$-face $C'$ in $\Asso(\nu')$. Therefore, the map $C\rightarrow C'$ where $C'=C\cup \{p_1,\dots,p_{a+b}\}$ is a bijection between $\Asso(\nu)$ and $\Asso(\nu')$.  
	It is straightforward to check that $\dim(C)=\dim(C')$, and so $\Asso(\nu)\cong\Asso(\nu')$. By definition, no lattice point in $A_{\nu'}$ in the prefix $N^{a}$ is relevant which implies $\relevantascent(C)=\relevantascent(C')$ and $\corelevant(C)=\corelevant(C')$.  The $F$-triangle thus remains unchanged.
\end{remark}

\section{Proof of the $F{=}H$ correspondence}
	\label{sec:main_proof}
In this section, we prove Theorem~\ref{thm:fh_correspondence}.  To illustrate this result, we reconsider our running example for $\nu=EENEN$.  We have
\begin{align*}
	x^{2}H_{EENEN}\left(\frac{x+1}{x},\frac{y+1}{x+1}\right) & = (y+1)^{2} + (x+1)(y+1) + (x+1)^{2} + 2x(y+1)\\
	& \kern1cm + 3x(x+1) + x^{2}\\
	& = y^{2}+2y+1 + xy+x+y+1 + x^{2}+2x+1 + 2xy\\
	& \kern1cm + 2x + 3x^{2}+3x + x^{2}\\
	& = 5x^{2} + 3xy + y^{2} + 8x + 3y + 3\\
	& = F_{EENEN}(x,y).
\end{align*}

Now, in general, if we plug in the definition of $H_{\nu}$ in \eqref{eq:h_to_f}, we obtain:
\begin{align*}
	x^{\deg(\nu)}H_{\nu}\left(\frac{x+1}{x},\frac{y+1}{x+1}\right) & = x^{\deg(\nu)}\sum_{\mu\in\Dyck_{\nu}}\left(\frac{x+1}{x}\right)^{\valley(\mu)}\left(\frac{y+1}{x+1}\right)^{\return(\mu)}\\
	& = \sum_{\mu\in\Dyck_{\nu}}x^{\deg(\nu)-\valley(\mu)}(x+1)^{\valley(\mu)-\return(\mu)}(y+1)^{\return(\mu)}.
\end{align*}

This is certainly a polynomial in $x$ and $y$ with nonnegative integer coefficients, because $\valley(\mu)\geq\return(\mu)$ and $\deg(\nu)=\max\bigl\{\valley(\mu)\mid\mu\in\Dyck_{\nu}\bigr\}$. Theorem~\ref{thm:fh_correspondence} is then equivalent to the following proposition. 

\begin{proposition}\label{prop_Ftriangle_paths}
	For every northeast path $\nu$, the following holds:
	\begin{equation}\label{eq:h_to_f_on_paths}
		F_{\nu}(x,y) = \sum_{\mu\in\Dyck_{\nu}}x^{\deg(\nu)-\valley(\mu)}(x+1)^{\valley(\mu)-\return(\mu)}(y+1)^{\return(\mu)}.
	\end{equation}
\end{proposition}

In order to prove this proposition we will first transform this expression to another expression in terms of $\nu$-trees, using the bijection from Section~\ref{sec_rightflushingbijection} (see Proposition~\ref{prop_Ftriangle_trees}). The second ingredient in our proof will be to show that the term associated with a $\nu$-tree in this new expression is equal to the sum of terms contributed by a specific group of faces in the definition of the $F$-triangle (see Proposition~\ref{prop_FtriangleT}). In order to shape our intuition, these groups are visualized (as shadowed groups) in Figure~\ref{fig:eenen_face} for our running example.

We start by explaining that the bijection $\Phi$ from $\nu$-paths to $\nu$-trees sends the valleys and returns to ascents and relevant nodes.  To do so, we define another statistic on $\nu$-trees; the \defn{horizontal distance} $\hroot$.  If $T\in\Trees_{\nu}$ and $p\in T$ is a node, then $\hroot_{T}(p)$ equals the number of horizontal edges in the unique path in $T$ connecting $p$ to the root.  For instance, in the $\nu$-tree of Figure~\ref{fig:right_flushing}, the node labeled $\overline{p}_{9}$ has horizontal distance $3$, counting the edges $\{\overline{p}_{9},\overline{p}_{10}\}$, $\{\overline{p}_{11},\overline{p}_{12}\}$, $\{\overline{p}_{12},\overline{p}_{13}\}$.  

\begin{lemma}\label{lem_changingstatistics}
	Let $\mu$ be a $\nu$-path and $T=\Phi(\mu)$ be its corresponding $\nu$-tree. Then,
	\begin{enumerate}[\rm (i)]
		\item \label{lem_changingstatistics_one} $\valley(\mu)=\ascent(T)$, \quad and 
		\item $\return(\mu)=\relevantascent(T)$.
	\end{enumerate}
\end{lemma}
\begin{proof}
	We denote by $p_0,p_1,\dots, p_\ell$ the lattice points of $\mu$ in the order they appear along the path. We denote by $\overline p_0, \overline p_1,\dots ,\overline p_\ell$ the nodes in $T$ ordered from the bottom row to the top row, and in each row from right to left. See Figure~\ref{fig:right_flushing} for an example.  
	Then, a point $p\in \mu$ is a valley if and only if its corresponding node $\overline p$ is an ascent of $T$. Therefore $\valley(\mu)=\ascent(T)$.

	For the second claim of the lemma, it is not hard to verify that 
	\begin{displaymath}
		\horiz_\nu(p_i) = \hroot_T(\overline p_i)
	\end{displaymath}
	for $0\leq i\leq\ell$. 
	Let $p$ be a return of $\mu$, \ie a valley of $\mu$ which is also a valley of~$\nu$.  In particular, it satisfies $\horiz_\nu(p)=0$. Therefore, the node~$\overline p\in T$ corresponding to~$p$ is in the first column because $\hroot_{T}(\overline{p})=0$.  Since $p$ is a valley of $\mu$, it follows that there must be another node $\overline{q}$ in the same row as $\overline{p}$.  
	Since $p$ is also valley of $\nu$, the row of $\overline{p}$ contains a valley of $\nu$. These three properties imply that $\overline{p}$ is relevant.  
	Vice-versa, if~$\overline p\in T$ is a relevant node, then we can similarly show that $p\in \mu$ must be a return of $\mu$.  Therefore, $\return(\mu)=\relevantascent(T)$.
\end{proof}

\begin{lemma}
	For a $\nu$-tree $T$, the following holds:
	\begin{align}
		(y+1)^{\relevantascent(T)} & = \sum_{A'\subseteq \Relevantascent(T)} y^{\relevantascent(T)-\lvert A'\rvert}\label{identityone},\\
		x^{\deg(\nu)-\ascent(T)}(x+1)^{\ascent(T)-\relevantascent(T)} & = \sum_{A''\subseteq \Ascent(T) \setminus \Relevantascent(T)} x^{\deg(\nu)-\relevantascent(T)-\lvert A''\rvert}\label{identitytwo}.
	\end{align}
\end{lemma}
\begin{proof}
	Since $\bigl\lvert\Relevantascent(T)\bigr\rvert=\relevantascent(T)$, Equation~\eqref{identityone} follows from the Binomial Theorem:
	\begin{displaymath}
		(y+1)^{\relevantascent(T)} = \sum_{k=0}^{\relevantascent(T)}\binom{\relevantascent(T)}{k}y^{\relevantascent(T)-k} = \sum_{A'\subseteq\Relevantascent(T)}y^{\relevantascent(T)-\lvert A'\rvert}
	\end{displaymath}
	Since $\lvert\Ascent(T) \setminus \Relevantascent(T)\rvert=\ascent(T)-\relevantascent(T)$, Equation~\eqref{identitytwo} can be shown similarly:
	\begin{align*}
		x^{\deg(\nu)-\ascent(T)}(x+1)^{\ascent(T)-\relevantascent(T)} & = x^{\deg(\nu)-\ascent(T)}\sum_{A''\subseteq\Ascent(T)\setminus\Relevantascent(T)}x^{\ascent(T)-\relevantascent(T)-\lvert A''\rvert}\\
		& = \sum_{A''\subseteq\Ascent(T)\setminus\Relevantascent(T)}x^{\deg(\nu)-\relevantascent(T)-\lvert A''\rvert}.\qedhere
	\end{align*}
\end{proof}

\begin{lemma}[\cite{ceballos19sweakorder}*{Lemma~5.4}]\label{lem_asso_characterization_faces}
	The sets $\bigl\{(T,A)\mid T\in\Trees_{\nu}, A\subseteq\Ascent(T)\bigr\}$ and $\Asso(\nu)$ are in bijection via the map $(T,A)\mapsto T\setminus A$.
	The dimension of the face $T\setminus A$ in the $\nu$-associahedron is $\dim(T\setminus A)=|A|$.
\end{lemma}

If $C\in\Asso(\nu)$ is of the form $C=T\setminus A$, then we say that $T$ is the \defn{bottom $\nu$-tree} of $C$.  This terminology is motivated as follows.  Recall that $\Asso(\nu)$ is a polytopal complex, so any face $C\in\Asso(\nu)$ is itself a polytope.  The edge graph of $C$ corresponds to an interval of $\Tamari(\nu)$ and as such inherits the orientation given by the partial order~$\leq_{\nu}$.  Then, $T$ is the minimal element of this interval.  Moreover, every ascent $p\in\Ascent(T)$ uniquely determines a $\nu$-tree $T_{p}$ with $T\lessdot_{\nu}T_{p}$; therefore the maximal $\nu$-tree in this interval is $T\vee\bigvee_{p\in\Ascent(T)}T_{p}$ (considered as a join in the lattice $\Tamari(\nu)$).

We denote by $\Cov_{T}(\nu)$ the set of covering $\nu$-faces whose bottom $\nu$-tree is $T$.  We define
\begin{equation}\label{eq:tree_contribution_separate}
	F_{\nu}^T(x,y) \defs \sum_{C\in \Cov_{T}(\nu)} x^{\deg(\nu)-\dim(C)-\relevantascent(C)} y^{\relevantascent(C)}.
\end{equation}
In our example in Figure~\ref{fig:eenen_face}, the sets $\Cov_{T}(\nu)$ are represented by the shadowed groups.
More precisely, the set $\Cov_{T}(\nu)$ consists of the faces belonging to the shadowed group containing $T$. 
The polynomial $F_{\nu}^T(x,y) $ is then the sum of the monomials in that shadowed group.
For instance, if $T_0$ is the bottom tree in Figure~\ref{fig:eenen_face}, then 
\begin{displaymath}
	F_{\nu}^{T_0}(x,y) = y^2+y+y+1 = (y+1)^2. 
\end{displaymath}
Compare the terms in Figure~\ref{fig:eenen_face} with the red ones (top expression per path) in Figure~\ref{fig:eenen_tamari_paths_fhlabel}.

\begin{proposition}\label{prop_FtriangleT}
	For every northeast path $\nu$, the following holds:
	\begin{equation}\label{eq:tree_contribution_full}
		F_{\nu}^T(x,y) = x^{\deg(\nu)-\ascent(T)}(x+1)^{\ascent(T)-\relevantascent(T)}(y+1)^{\relevantascent(T)}.
	\end{equation}
\end{proposition}
\begin{proof}
	Let $C\in \Cov_{T}(\nu)$. Then $C=T\setminus A$ for some subset $A$ of ascents of $T$. This subset can be written uniquely as a disjoint union $A=A'\uplus A''$, where $A'\subseteq \Relevantascent(T)$ and $A''\subseteq \Ascent(T) \setminus \Relevantascent(T)$.
	Then $\relevantascent(C) =\relevantascent(T)-\lvert A'\rvert$.  Furthermore $\dim(C)=\lvert A'\rvert+\lvert A''\rvert$, and so $\deg(\nu)-\dim(C)-\relevantascent(C) = \deg(\nu)-\relevantascent(T)-\lvert A''\rvert$. Therefore, 
	\begin{align*}
		F_{\nu}^T(x,y) & = \sum_{C\in \Cov_{T}(\nu)} x^{\deg(\nu)-\dim(C)-\relevantascent(C)} y^{\relevantascent(C)}\\
			& = \sum_{A=A'\uplus A''} x^{\deg(\nu)-\relevantascent(T)-\lvert A''\rvert} y^{\relevantascent(T)-\lvert A'\rvert}.
	\end{align*}
	This is exactly the product of Equations \eqref{identityone} and~\eqref{identitytwo}, and the result follows.
\end{proof}

\begin{proposition}\label{prop_Ftriangle_trees}
	For every northeast path $\nu$, the following holds:
	\begin{equation}
		F_{\nu}(x,y) = \sum_{T\in\Trees_{\nu}}x^{\deg(\nu)-\ascent(T)}(x+1)^{\ascent(T)-\relevantascent(T)}(y+1)^{\relevantascent(T)}.
	\end{equation}
\end{proposition}
\begin{proof}
	Since $\Cov(\nu) =\biguplus_{T\in\Trees_{\nu}} \Cov_{T}(\nu)$, it follows that
	\begin{align*}
		F_{\nu}(x,y) &= \sum_{T\in \Trees_\nu} F_{\nu}^T(x,y).
	\end{align*}
	The result then follows from Proposition~\ref{prop_FtriangleT}
\end{proof}

\begin{proof}[Proof of Proposition~\ref{prop_Ftriangle_paths}]
	Proposition~\ref{prop_Ftriangle_paths} follows from Proposition~\ref{prop_Ftriangle_trees} and Lemma~\ref{lem_changingstatistics}, by transforming the statistics under the right flushing bijection $\Phi$.
\end{proof}

\begin{proof}[Proof of Theorem~\ref{thm:fh_correspondence}]
	As we have already mentioned, Theorem~\ref{thm:fh_correspondence} is equivalent to Proposition~\ref{prop_Ftriangle_paths}, which we have just proven.
\end{proof}

\begin{example}\label{ex_3dAsso}
	We finish this section by presenting a 3-dimensional example of our results. For this we consider the path $\nu=ENEENEN$. The corresponding~$\nu$-associahedron $\Asso(\nu)$ is illustrated in Figure~\ref{fig:f_triangle_explanation}, whose vertices are labeled by $\nu$-paths and the terms they contribute to the $F$-triangle $F_{\nu}$ (top expression in red obtained from Proposition~\ref{prop_FtriangleT}) and to the $H$-triangle $H_{\nu}$ (bottom expression in blue obtained from~\eqref{eq:h_triangle}). Summing over these labels yields:

	\begin{align*}
		F_{ENEENEN}(x,y) & = 9x^3 + 9x^2y + 4xy^2 + y^3 + 20x^2 + 15xy + 4y^2 + 14x + 6y + 3,\\
		H_{ENEENEN}(x,y) & = x^3y^3 + x^3y^2 + x^3y + 3x^2y^2 + 5x^2y + 3x^2 + 3xy + 5x + 1.
	\end{align*}

	The explicit computation of the terms $F_{\nu}^T(x,y)$, using the definition in \eqref{eq:tree_contribution_separate}, is shown for the two green faces in the figure. This illustrates two examples of Proposition~\ref{prop_FtriangleT}. 
\end{example}

\begin{figure}
	\centering
	\includegraphics[width=\textwidth,page=10]{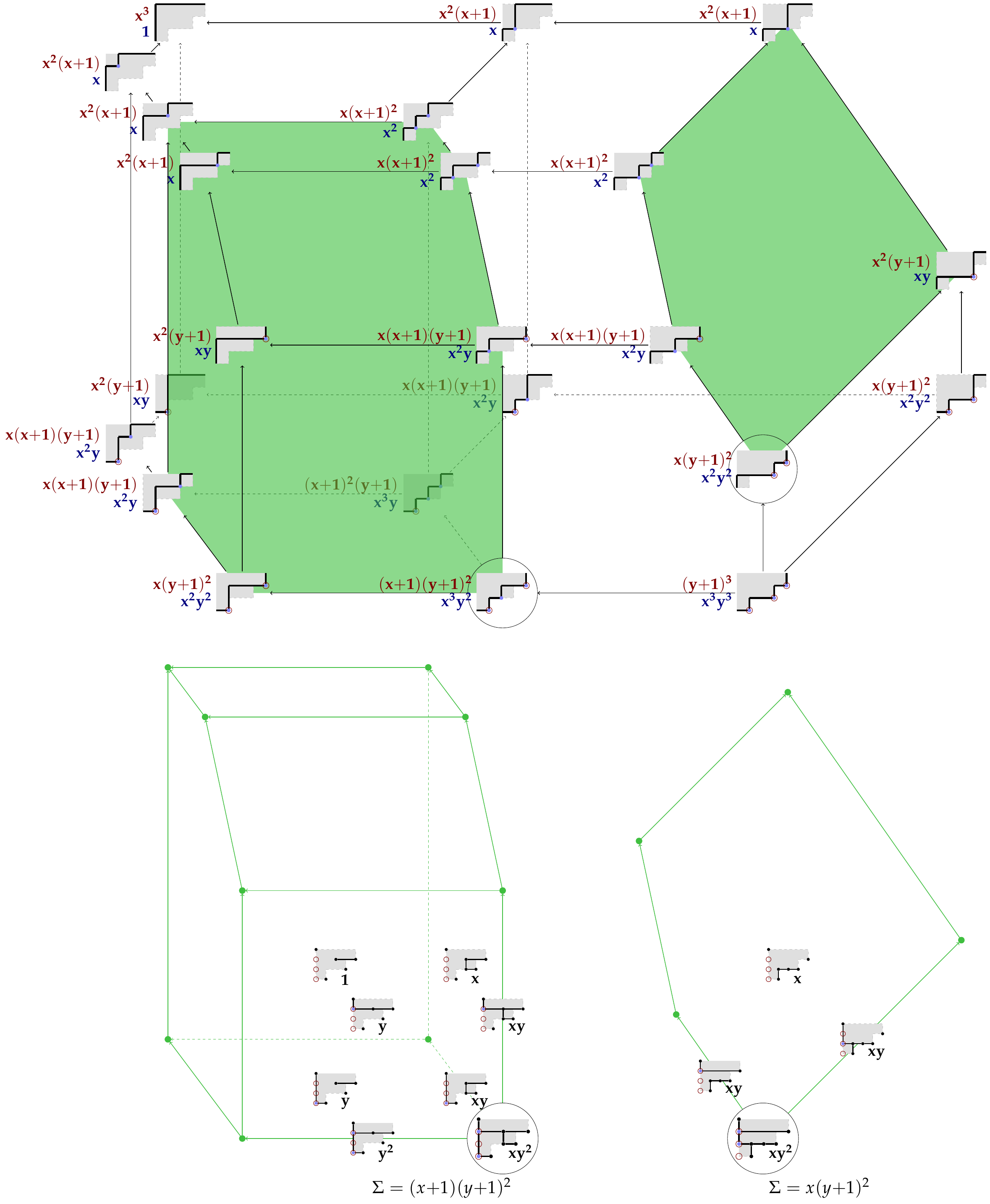}
	\caption{The $\nu$-associahedron for $\nu=ENEENEN$, where the vertices are labeled by $\nu$-paths together with the terms they contribute to $F_{\nu}$ (top expression in red) and $H_{\nu}$ (bottom expression in blue).  The green faces are additionally labeled with the faces in $\Cov_{T}(\nu)$, where $T$ is each time the image under the right-flushing bijection of the associated bottom $\nu$-path (circled).  The sum over the expressions contributed by these faces equals the term that is contributed by the path; illustrating Proposition~\ref{prop_FtriangleT}.}
	\label{fig:f_triangle_explanation}
\end{figure}

\section{Revisiting the $\nu$-Tamari complex}
	\label{sec:nu_tamari_reciprocity}
In this section we outline that the specialization at $y=x$ of \eqref{eq:h_to_f} can be seen as a key component of a certain reciprocity result.  We outline here the situation only for the $\nu$-Tamari complex, and refer the reader to \cite{ceballos21revisiting} for an explanation of this ``$f{=}h$ reciprocity'' for a more general class of simplicial complexes and a discussion of its relation to various generalizations of the Dehn--Sommerville relations.

Consider the $\nu$-Tamari complex $\TamariComplex({\nu})$ associated with a northeast path $\nu$ from $(0,0)$ to $(m,n)$. 
Recall that the facets of $\TamariComplex({\nu})$ are given by $\nu$-trees, and every $\nu$-tree contains $m+n+1$ elements. 
Therefore $\TamariComplex({\nu})$ is a pure simplicial complex of dimension
\begin{displaymath}
	d' -1 \defs \dim \bigl(\TamariComplex({\nu})\bigr) = m+n.
\end{displaymath}
Since it can be realized as a triangulation of a polytope~\cite{ceballos19geometry}, it is actually a $(d'{-}1)$-ball. 

If $\mathcal{C}$ is a polytopal complex, its \defn{$f$-vector} consists of the numbers $f_{i}(\mathcal{C})$ counting the faces of dimension $i$.  If $\mathcal{C}$ has dimension\footnote{The dimension of a polytopal complex $\mathcal{C}$ is defined as the largest dimension of the polytopes in $\mathcal{C}$.} $d$, then we may define its \defn{$h$-vector} to consist of the numbers $h_{i}(\mathcal{C})$ defined through the following base change
\begin{displaymath}
	\sum_{i=0}^{d}h_{i}(\mathcal{C})x^{i} = \sum_{i=0}^{d}f_{i-1}(\mathcal{C})x^{i}(1-x)^{d-i}.
\end{displaymath}
Moreover, we consider the following polynomials associated with these vectors:
\begin{align*}\label{ftilde_htilde_polynomial}
	f_{\mathcal{C}}(x) \defs \sum_{i=0}^d f_{i-1}(\mathcal{C})x^{d-i}, &&
	\tilde f_{\mathcal{C}}(x) \defs \sum_{i=0}^d f_{i-1}(\mathcal{C})x^{i},\\	
	h_{\mathcal{C}}(x) \defs \sum_{i=0}^d h_{i}(\mathcal{C})x^{d-i}, &&
	\tilde h_{\mathcal{C}}(x) \defs \sum_{i=0}^d h_{i}(\mathcal{C})x^{i}.
\end{align*}

As we can see from the definition of the $F$-triangle in \eqref{eq:f_triangle}, the polynomial $F_{\nu}(x,y)$ is a bivariate generating function of the faces of $\Asso(\nu)$. Evaluating this polynomial at $y=x$ recovers the face numbers of the $\nu$-associahedron.

\begin{proposition}\label{prop_F_f_nu_associahedron}
	The $\nu$-associahedron $\Asso(\nu)$ is a polytopal complex of dimension $d=\deg(\nu)$, and we have
	\begin{displaymath}
%		F_{\nu}(x,x) = f_{d} + f_{d-1}x + \dots + f_1x^{d-1} + f_0x^d = \sum_{i=0}^{d} f_i x^{d-i}.
		F_{\nu}(x,x) = f_{\Asso(\nu)}(x).
	\end{displaymath}
\end{proposition}
\begin{proof}
	We start by proving that $\dim\bigl(\Asso(\nu)\bigr)= \deg(\nu)$. Let $\nu$ be a northeast path from $(0,0)$ to~$(m,n)$. The cells of $\Asso(\nu)$ are covering $\nu$-faces C ordered by reversed inclusion, whose dimensions are $\dim(C) = m+n+1 - \lvert C\rvert$.
	The dimension of a $\nu$-tree $T$ is equal to zero (because it contains $m+n+1$ elements). Any other cell is of the form $C=T\setminus A$ for some $\nu$-tree $T$ and a subset $A$ of ascents of $T$ (Lemma~\ref{lem_asso_characterization_faces}), in which case $\dim (C) = |A|$.
	Recall that $\deg(\nu)$ is the maximal number of valleys a $\nu$-path can have. By Lemma~\ref{lem_changingstatistics}(\ref{lem_changingstatistics_one}), this is also equal to the maximal number of ascents a $\nu$-tree can have, and therefore is the largest dimension of a cell of $\Asso(\nu)$. Thus, the dimension of $\Asso(\nu)$ is equal to $\deg(\nu)$.

	For the second part of the result, we have that
	\begin{align*}
		F_{\nu}(x,x) &= \sum_{C\in \Cov(\nu)} x^{\corelevant(C)+\relevantascent(C)} = \sum_{C\in \Cov(\nu)} x^{\deg(\nu)-\dim(C)}.
	\end{align*}
	Each cell of dimension $i$ contributes a term $x^{\deg(\nu)-i}$ to this sum, and the result follows. 
\end{proof}

\begin{example}
	Continuing Example~\ref{ex_3dAsso} for $\nu=ENEENEN$, we get 
	\begin{displaymath}
		F_{ENEENEN}(x,x) = 23x^3+39x^2+ 20x+3.
	\end{displaymath}
	The coefficients of this polynomial count the number of faces of the $\nu$-associahedron in Figure~\ref{fig:eneenen_associahedron}:
	it has $23$ vertices, $39$ edges, $20$ two-dimensional faces, and $3$ three-dimensional faces.
\end{example}

Using a shelling order of the $\nu$-Tamari complex $\TamariComplex(\nu)$,
it was proven in~\cite{ceballos19geometry}*{Theorem~4.6} that the $h$-vector $(h'_0,h'_1,\ldots)$ of $\TamariComplex(\nu)$ satisfies that $h'_i$ is equal to the number of $\nu$-paths with exactly $i$ valleys (the $h'_{i}$'s are known as the \defn{$\nu$-Narayana numbers}). As a consequence, the $H$-triangle $H_\nu(x,y)$ is a bivariate polynomial generalization of the $\tilde h$-polynomial of $\TamariComplex(\nu)$. 

\begin{proposition}\label{prop_H_hTamariComplex}
	Let $(h'_0,h'_1,\ldots)$ be the $h$-vector of the $\nu$-Tamari complex $\TamariComplex(\nu)$. Then $h'_i>0$ for $0\leq i \leq \deg(\nu)$ and $h'_i=0$ otherwise. Moreover,
	\begin{equation}\label{eq_h_vector_TamariComplex}
		H_{\nu}(x,1) = \tilde h_{\TamariComplex(\nu)}(x).
	\end{equation}
\end{proposition}
\begin{proof}
	By definition, $H_{\nu}(x,1) = \sum_{\mu\in\Dyck_{\nu}}x^{\valley(\mu)}$.
	Each $\nu$-path with exactly $i$ valleys contributes a term $x^i$ to this sum. Since there are $h'_i$ $\nu$-paths with $i$ valleys~\cite{ceballos19geometry}*{Theorem~4.6}, Equation~\eqref{eq_h_vector_TamariComplex} follows. Since the maximum number of valleys of a $\nu$-path is $\deg(\nu)$ we get $h'_i=0$ whenever $i>\deg(\nu)$. 
\end{proof}

\begin{remark}
	We remark that the dimension of the $\nu$-Tamari complex $\TamariComplex(\nu)$ is one less than the number of elements of a $\nu$-tree. This number is much larger than the dimension of the $\nu$-associahedron $\Asso(\nu)$. For instance, for $\nu=EENEN$ (Figure~\ref{fig:eenen_face}) we have $\dim\bigl(\TamariComplex(\nu)\bigr)=6-1=5$ while $\dim\bigl(\Asso(\nu)\bigr)=2$. For $\nu=ENEENEN$ (Figure~\ref{fig:f_triangle_explanation}), we have $\dim\bigl(\TamariComplex(\nu)\bigr)=8-1=7$ while $\dim\bigl(\Asso(\nu)\bigr)=3$. 

	Although the $\nu$-Tamari complex is a ``more complicated" object due to its high dimension, its sub-complex of interior faces (which is dual to the $\nu$-associahedron) is simpler and retains a lot of information about its structure. For instance, the $F{=}H$ correspondence (Theorem~\ref{thm:fh_correspondence}) together with the results in this section tell us that we can recover the $h$-vector of $\TamariComplex(\nu)$ in terms of the $f$-vector of $\Asso(\nu)$. This implies, in particular, that we can obtain the $f$-vector of $\TamariComplex(\nu)$ in terms of the $f$-vector of $\Asso(\nu)$ as well. 
\end{remark}

The $\nu$-associahedron $\Asso(\nu)$ is defined as the dual complex of interior faces of~$\TamariComplex({\nu})$ under reverse inclusion~\cite{ceballos19geometry}.  If we denote the number of such interior faces of dimension $i$ by $f_{i}^{\interior}\bigl(\TamariComplex(\nu)\bigr)$, we get the following correspondence:
\begin{align}\label{eq_facenumbers_duality}
	f_{i-1}^\interior\bigl(\TamariComplex({\nu})\bigr) = f_{d'-i} \bigl(\Asso(\nu)\bigr).
\end{align}
Using Theorem~\ref{thm:fh_correspondence} specialized at $y=x$, we obtain
\begin{align*}
	x^{d'}\tilde{h}_{\TamariComplex(\nu)}\left(\frac{x+1}{x}\right) & = x^{d'}H_{\nu}\left(\frac{x+1}{x},1\right) & \text{(by Proposition~\ref{prop_H_hTamariComplex})}\\
	& = x^{d'-\deg(\nu)}F_{\nu}(x,x) & \text{(by Theorem~\ref{thm:fh_correspondence})}\\
	& = x^{d'-\deg(\nu)}\sum_{j=0}^{\deg(\nu)}f_{j}\bigl(\Asso(\nu)\bigr)x^{\deg(\nu)-j} & \text{(by Proposition~\ref{prop_F_f_nu_associahedron})}\\
	& = \sum_{j=0}^{\deg(\nu)}f_{j}\bigl(\Asso(\nu)\bigr)x^{d'-j} \\
	& = \sum_{j=0}^{d'-1}f_{j}\bigl(\Asso(\nu)\bigr)x^{d'-j}
\end{align*}
The last equation follows from the fact that $\dim\bigl(\Asso(\nu)\bigr)=\deg(\nu)$ and we have $f_{j}\bigl(\Asso(\nu)\bigr)=0$ for $j>\deg(\nu)$.  By reversing the order of summation, we can rewrite the last equation as follows:
\begin{align*}
	x^{d'}\tilde{h}_{\TamariComplex(\nu)}\left(\frac{x+1}{x}\right) & =\sum_{i=1}^{d'}f_{d'-i}\bigl(\Asso(\nu)\bigr)x^{i}\\
	& = \sum_{i=1}^{d'}f_{i-1}^{\interior}\bigl(\TamariComplex(\nu)\bigr)x^{i} & \text{(by Equation~\eqref{eq_facenumbers_duality})}\\
	& = \tilde f_{\TamariComplex(\nu)}^{\interior}(x),
\end{align*}
where the last equality follows because the empty face is (by definition) not interior.  We call the equality 
\begin{equation}\label{eq:tamaricomplex_fh_reciprocity}
	x^{d'}\tilde{h}_{\TamariComplex(\nu)}\left(\frac{x+1}{x}\right) = \tilde f_{\TamariComplex(\nu)}^{\interior}(x)
\end{equation}
the ``$f{=}h$ reciprocitiy for $\TamariComplex(\nu)$''.  
This relation says that the evaluation of the $\widetilde h$-polynomial of the $\nu$-Tamari complex on the left side counts the number of interior faces of the complex.  
In \cite{ceballos21revisiting},  we prove a generalization of this result for a large class of simplicial complexes which we call \defn{reciprocal complexes}, and use it to revisit several generalizations of the Dehn--Sommerville relations.

\section{A multivariate $F{=}H$-correspondence for finite posets}
	\label{sec:fh_correspondence_posets}
If we have a closer look at the underlying mechanics of the $F{=}H$-correspondence for $\nu$-associahedra explained in Section~\ref{sec:main_proof}, then we notice that the key observation is the distinction between two types of edges in the $\nu$-associahedron.  Indeed, in Lemma~\ref{lem_asso_characterization_faces}, the faces of $\Asso(\nu)$ are identified with pairs $(T,A)$, where $T$ is a $\nu$-tree and $A$ is a subset of the ascent nodes of $T$.  In Proposition~\ref{prop_FtriangleT}, the contribution of such a pair $(T,A)$ to $F_{\nu}(x,y)$ is described by partitioning $A$ into relevant and non-relevant ascents.  By definition of the rotation order on trees in Section~\ref{sec:nu_tamari_trees}, every ascent node of $T$ corresponds to a unique $\nu$-tree $T'$ such that $T\lessdot_{\nu}T'$.

It is not too far a stretch to distinguish edges in $\Tamari(\nu)$ with respect to rotating by a relevant or a non-relevant ascent.  From this perspective, we may define analogues of $F$- and $H$-triangles for arbitrary (finite) posets in such a way that we retain the $F{=}H$-correspondence.  In fact, there is no need to restrict ourselves to only \emph{two} types of edges.

Let $\Poset=(P,\leq)$ be a finite partially ordered set (or \defn{poset}).  We write $p\lessdot q$ if $p<q$ and there does not exist $r\in P$ such that $p<r<q$.  Such relations are \defn{cover relations} (or edges) of $\Poset$, and we write 
\begin{displaymath}
	\Covers(\Poset)\defs\bigl\{(p,q)\in P\times P\mid p\lessdot q\bigr\}.
\end{displaymath}
For an integer $k>0$, a \defn{$k$-valued edge labeling} of $\Poset$ is a map $\lambda\colon\Covers(\Poset)\to\{0,1,\ldots,k\}$.  For $p\in P$, we define its set of \defn{successors} by
\begin{displaymath}
	\Succ(p) \defs \bigl\{p'\in P\mid p\lessdot p'\bigr\}.
\end{displaymath}

Given a finite poset $\Poset=(P,\leq)$ with $k$-valued edge labeling $\lambda$, we consider the following statistics where $i\in[k]$:
\begin{align*}
	\out(p) & \defs \bigl\lvert\Succ(p)\bigr\rvert,\\
	\mrk_{i}(p) & \defs \bigl\lvert\bigl\{p'\in \Succ(p)\mid \lambda(p,p')=i\bigr\}\bigr\rvert.
\end{align*}
Then, the \defn{$(k+1)$-variate $H$-triangle} of $\Poset$ (with respect to $\lambda$) is
\begin{displaymath}
	H_{\Poset,\lambda}(x,y_{1},\ldots,y_{k}) \defs \sum_{p\in P}x^{\out(p)}y_{1}^{\mrk_{1}(p)}\cdots y_{k}^{\mrk_{k}(p)}.
\end{displaymath}

The \defn{degree} of $\Poset$ is $\deg(\Poset)\defs\max\bigl\{\out(p)\mid p\in P\bigr\}$.  For $p\in P$, $S\subseteq\Succ(p)$ and $i\in[k]$, we define\begin{align*}
	\rem_{i}(p,S) & \defs \mrk_{i}(p) - \bigl\lvert\bigr\{s\in S\mid \lambda(p,s)=i\bigr\}\bigr\rvert,\\
	\corem(p,S) & \defs \deg(\Poset) - \lvert S\rvert - \rem_{1}(p,S) - \rem_{2}(p,S) - \cdots - \rem_{k}(p,S).
\end{align*}
More precisely, $\rem_{i}(p,S)$ counts the outgoing edges labeled by $i$ which are not selected (hence ``remain in $p$'') by $S$.

The \defn{$(k+1)$-variate $F$-triangle} of $\Poset$ (with respect to $\lambda$) is
\begin{displaymath}
	F_{\Poset,\lambda}(x,y_{1},\ldots,y_{k}) \defs \sum_{p\in P}\sum_{S\subseteq\Succ(p)}x^{\corem(p,S)}y_{1}^{\rem_{1}(p,S)}\cdots y_{k}^{\rem_{k}(p,S)}.
\end{displaymath}

\begin{theorem}\label{thm:fh_poset_correspondence}
	For every finite poset $\Poset$ and every $k$-valued edge labeling $\lambda$ of $\Poset$, we have
	\begin{equation}\label{eq:multi_fh}
		F_{\Poset,\lambda}(x,y_{1},\ldots,y_{k}) = x^{\deg(\Poset)}H_{\Poset,\lambda}\left(\frac{x+1}{x},\frac{y_{1}+1}{x+1},\ldots,\frac{y_{k}+1}{x+1}\right).
	\end{equation}
	Equivalently,
	\begin{multline}\label{eq:multi_hf}
		H_{\Poset,\lambda}(x,y_{1},\ldots,y_{k})\\
			= (x-1)^{\deg(\Poset)}F_{\Poset,\lambda}\left(\frac{1}{x-1},\frac{x(y_{1}{-}1){+}1}{x-1},\ldots,\frac{x(y_{k}{-}1){+}1}{x-1}\right).
	\end{multline}
\end{theorem}
\begin{proof}
	We prove \eqref{eq:multi_fh}, and consider
	\begin{align*}
		\mrk_{0}(p) & = \bigl\lvert\bigl\{p'\in\Succ(p)\mid \lambda(p,p')=0\bigr\}\bigr\rvert\\
		& = \out(p) - \mrk_{1}(p) - \cdots - \mrk_{k}(p).
	\end{align*}
	Moreover, for $i\in\{0,1,\ldots,k\}$ we write
	\begin{displaymath}
		S_{i}(p) \defs \bigl\{p'\in\Succ(p)\mid \lambda(p,p')=i\bigr\}.
	\end{displaymath}
	Then, $\mrk_{i}(p)=\bigl\lvert S_{i}(p)\bigr\rvert$, and $\Succ(p)=\biguplus_{i=0}^{k}S_{i}(p)$.  For $S\subseteq\Succ(p)$ the notation ``$S=A_{0}\uplus A_{1}\uplus\cdots\uplus A_{k}$'' is meant to describe the decomposition of $S$, where $A_{i}=S\cap S_{i}(p)$.  For such a decomposition, we have $\rem_{i}(p,S)=\mrk_{i}(p)-\lvert A_{i}\rvert$.  Moreover, we have 
	\begin{align*}
		\corem(p,S) & = \deg(\Poset) - \lvert S\rvert - \rem_{1}(p,S) - \rem_{2}(p,S) - \cdots - \rem_{k}(p,S)\\
		& = \deg(\Poset) - \lvert A_{0}\rvert - \lvert A_{1}\rvert - \cdots - \lvert A_{k}\rvert - \rem_{1}(p,S) - \cdots - \rem_{k}(p,S)\\
		& = \deg(\Poset) - \lvert A_{0}\rvert - \mrk_{1}(p) - \cdots - \mrk_{k}(p)\\
		& = \deg(\Poset) - \out(p) + \mrk_{0}(p) - \lvert A_{0}\rvert.
	\end{align*}
	We therefore obtain
	\begin{align*}
		x^{\deg(\Poset)}H_{\Poset,\lambda} & \left(\frac{x+1}{x},\frac{y_{1}+1}{x+1},\ldots,\frac{y_{k}+1}{x+1}\right)\\
		& = x^{\deg(\Poset)}\sum_{p\in P}\left(\frac{x+1}{x}\right)^{\out(p)}\prod_{i=1}^{k}\left(\frac{y_{i}+1}{x+1}\right)^{\mrk_{i}(p)}\\
		& = \sum_{p\in P}x^{\deg(\Poset)-\out(p)}(x+1)^{\mrk_{0}(p)}\prod_{i=1}^{k}(y_{i}+1)^{\mrk_{i}(p)}\\
		& = \sum_{p\in P}x^{\deg(\Poset)-\out(p)}\sum_{A_{0}\subseteq S_{0}(p)}x^{\mrk_{0}(p)-\lvert A_{0}\rvert}\prod_{i=1}^{k}\sum_{A_{i}\subseteq S_{i}(p)}y_{i}^{\mrk_{i}(p)-\lvert A_{i}\rvert}\\
		& = \sum_{p\in P}\sum_{\substack{S\subseteq \Succ(p)\\S=A_{0}\uplus A_{1}\uplus\cdots\uplus A_{k}}}x^{\deg(\Poset)-\out(p)+\mrk_{0}(p)-\lvert A_{0}\rvert}\prod_{i=1}^{k}y_{i}^{\mrk_{i}(p)-\lvert A_{i}\rvert}\\
		& = \sum_{p\in P}\sum_{S\subseteq \Succ(p)}x^{\corem(p,S)}\prod_{i=1}^{k}y_{i}^{\rem_{i}(p,S)}\\
		& = F_{\Poset,\lambda}(x,y_{1},\ldots,y_{k}).
	\end{align*}

	Equation~\eqref{eq:multi_hf} follows by an appropriate substitution of variables.
\end{proof}

\begin{figure}
	\centering
	\includegraphics[scale=1,page=18]{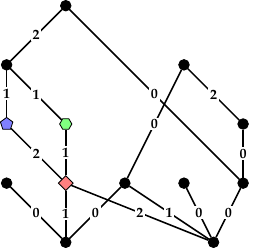}
	\caption{A poset with a $2$-valued edge labeling.}
	\label{fig:fh_poset}
\end{figure}

\begin{example}
	Let us consider the poset $\Poset$ with its $2$-valued labeling $\lambda$ shown in Figure~\ref{fig:fh_poset}.  We have $\deg(\Poset)=4$, because the minimal element on the right has four upper covers and no element has more.  If $p$ is the element displayed as a red lozenge, then we have $\out(p)=2$, $\mrk_{1}(p)=1$ and $\mrk_{2}(p)=1$.  Thus, $p$ contributes a term $x^{2}y_{1}y_{2}$ to $H_{\Poset,\lambda}(x,y_{1},y_{2})$.  Summing over all poset elements yields
	\begin{displaymath}
		H_{\Poset,\lambda}(x,y_{1},y_{2}) = x^{4}y_{1}y_{2} + x^{3}y_{1} + x^{2}y_{1}y_{2} + x^{2} + 2xy_{1} + 2xy_{2} + x + 4.
	\end{displaymath}
	The set $\Succ(p)$ consists of the two elements displayed as a blue pentagon and a green hexagon, respectively. If $S=\{q\}\subseteq\Succ(p)$, where $q$ is the blue pentagon, then we have $\lambda(p,q)=2$, and therefore $\rem_{1}(p,S)=\mrk_{1}(p)-0=1$ and $\rem_{2}(p,S)=\mrk_{2}(p)-1=0$.  We get $\corem(p,S)=4-1-1-0=2$.  Thus, the pair $(p,S)$ contributes the term $x^{2}y_{1}$ to $F_{\Poset,\lambda}(x,y_{1},y_{2})$.  Summing over all such pairs yields
	\begin{multline*}
		F_{\Poset,\lambda}(x,y_{1},y_{2}) = 6x^{4} + 3x^{3}y_{1} + 2x^{3}y_{2} + 2x^{2}y_{1}y_{2} + 8x^{3} + 4x^{2}y_{1} \\
			+ 2x^{2}y_{2} + 2xy_{1}y_{2} + 5x^{2} + 3xy_{1} + 2xy_{2} + y_{1}y_{2} + 3x + y_{1} + y_{2} + 1.
	\end{multline*}
	The reader is invited to verify that \eqref{eq:multi_fh} and \eqref{eq:multi_hf} are satisfied by these two polynomials.
\end{example}

The multivariate $F{=}H$-correspondence in \eqref{eq:multi_hf} was perhaps first stated in \cite{garver20chapoton}*{Theorem~5.4} for so-called Grid--Tamari lattices.  The definition of their $F$- and $H$-triangle depends on two simplicial complexes associated with the Grid--Tamari lattices.
We have not checked the details, but we are fairly confident about the existence of an appropriate edge-labeling of the Grid--Tamari lattices that recovers their $F$- and $H$-triangles in our setting.

\begin{remark}\label{rem:nu_tamari_marking}
	Let $\Poset=\Tamari(\nu)$, $k=1$ and $\lambda(T,T')=1$ if and only if $T'$ is obtained from $T$ by rotating at a relevant node.  We call this the \defn{rotation marking}.  Then, for every $\nu$-tree $T$ and every $A\subseteq\Ascent(T)$, we get $\rem_{1}(T,A)=\relevantascent(C)$ for $C=T\setminus A$ and $\corem(T,A)=\corelevant(C)$.  It follows right away that $F_{\Poset,\lambda}(x,y)=F_{\nu}(x,y)$ and $H_{\Poset,\lambda}(x,y)=H_{\nu}(x,y)$.
	Therefore, Theorem~\ref{thm:fh_correspondence} is a special case of Theorem~\ref{thm:fh_poset_correspondence}.
\end{remark}

\begin{remark}
    Remark~\ref{rem:nu_tamari_marking} also indicates how to obtain a multivariate version of Theorem~\ref{thm:fh_correspondence} for the $\nu$-Tamari lattices.  Suppose that $\nu$ has $k$ valleys, say in rows $i_{1},i_{2},\ldots,i_{k}$.  Then, let $\lambda(T,T')=j$ whenever $T'$ is obtained from $T$ by (i) rotating at a relevant node in row $i_{j}$ (when $j>0$) or (ii) rotating at a non-relevant node (when $j=0$).  For every $\nu$-tree $T$ and every $A\subseteq\Ascent(T)$, the statistic $\rem_{j}(T,A)$ records whether $C=T\setminus A$ contains the relevant node in row $i_{j}$.
\end{remark}

The perspective offered by Theorem~\ref{thm:fh_poset_correspondence} raises interesting questions.  On the one hand, Chapoton's original definition of the $F$- and $H$-triangles in \cites{chapoton06sur} is in the context of cluster complexes and root posets of crystallographic Coxeter groups.  Specializing our Theorem~\ref{thm:fh_correspondence} to the staircase path $\nu=(EN)^{n}$ recovers and explains Chapoton's construction in type $A_{n}$ with the classical Tamari lattice playing a key role.  The Tamari lattice generalizes to the other crystallographic Coxeter groups as a \defn{Cambrian lattice}, see \cites{reading06cambrian,reading07sortable}.  We pose the following research challenge, where we refer to \cite{reading07sortable} for any undefined notation.

\begin{question}\label{qu:cambrian_marking}
	Let $W$ be a finite, irreducible, crystallographic Coxeter group and let $c\in W$ be a Coxeter element.  Let $\mathbf{C}$ denote the corresponding Cambrian lattice.  Find and explain a $1$-valued edge labeling of $\mathbf{C}$ such that the corresponding $2$-variate $F$- and $H$-triangles recover Chapoton's original $F$- and $H$-triangles associated with~$W$.
\end{question}

\begin{figure}
	\centering
	\begin{subfigure}[t]{.4\textwidth}
		\centering
		\includegraphics[scale=.85,page=20]{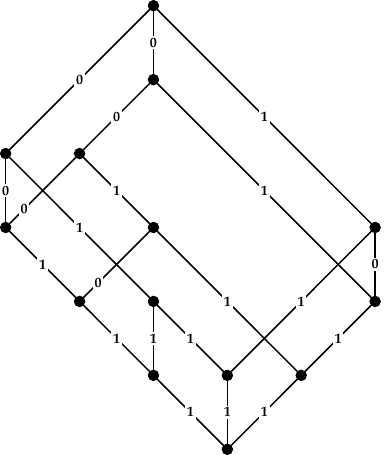}
		\caption{The rotation marking.}
		\label{fig:cambrian_1}
	\end{subfigure}
	\hspace*{1cm}
	\begin{subfigure}[t]{.4\textwidth}
		\centering
		\includegraphics[scale=.85,page=21]{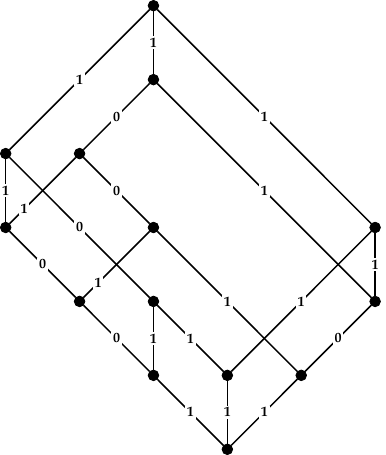}
		\caption{The atom-perspectivity marking.}
		\label{fig:cambrian_2}
	\end{subfigure}
	\caption{Two different markings of the Tamari lattice.}
\end{figure}

In particular, different Cambrian lattices associated with $W$ orient the \emph{same} polytope, the \defn{$W$-associahedron}.  Therefore, the associated $F$-triangle, enumerating the faces of the $W$-associahedron is the same for all Cambrian lattices of $W$.  This implies that we want markings of different, not necessarily isomorphic, Cambrian lattices to produce the \emph{same} $H$-triangle.

Let us illustrate Question~\ref{qu:cambrian_marking} with an example.  Figure~\ref{fig:cambrian_1} shows a lattice isomorphic to the $ENENEN$-Tamari lattice.  This is itself a Cambrian lattice associated with the Coxeter group $A_{3}$.  The marking displayed is the rotation-marking described in Remark~\ref{rem:nu_tamari_marking}.  We obtain the following $H$-triangle:
\begin{displaymath}
	H_{3}(x,y) = x^{3}y^{3} + 3x^{2}y^{2} + 2x^{2}y + x^{2} + 3xy + 2x + 1.
\end{displaymath}
Another, reasonably natural marking of the Cambrian lattices can be defined as follows.  If $\Lattice$ is a finite lattice, then two cover relations $(a_{1},b_{1})$ and $(a_{2},b_{2})$ are \defn{perspective} if either $a_{1}\vee b_{2}=b_{1}$ and $a_{1}\wedge b_{2}=a_{2}$ or $b_{1}\vee a_{2}=b_{2}$ and $b_{1}\wedge a_{2}=a_{1}$.  The \defn{atom-perspectivity marking} of $\Lattice$ marks $(a,b)\in\Covers(\Lattice)$ if and only if $(a,b)$ is perspective to some $(\hat{0},c)$, where $\hat{0}$ denotes the least element of $\Lattice$.  For the Tamari lattice, this marking is displayed in Figure~\ref{fig:cambrian_2}.  The associated $H$-triangle is
\begin{displaymath}
	\tilde{H}_{3}(x,y) = x^{3}y^{3} + 3x^{2}y^{2} + 3x^{2}y + 3xy + 3x + 1,
\end{displaymath}
which is different from $H_{3}(x,y)$.  This is also evident, because the rotation-marking marks $14$ cover relations, the atom-perspectivity marking marks $15$.  Figure~\ref{fig:cambrian_3} shows a non-Tamari Cambrian lattice associated with $A_{3}$ marked with the atom-perspectivity marking.  The reader is invited to check that this marking produces once again the correct $H$-triangle $H_{3}(x,y)$.  Computer experiments (up to $n=7$) suggest that in type $A$ this marking always produces the correct $H$-triangle precisely when $c$ is a \emph{bipartite} Coxeter element.

\begin{figure}
	\centering
	\includegraphics[scale=.85,page=22]{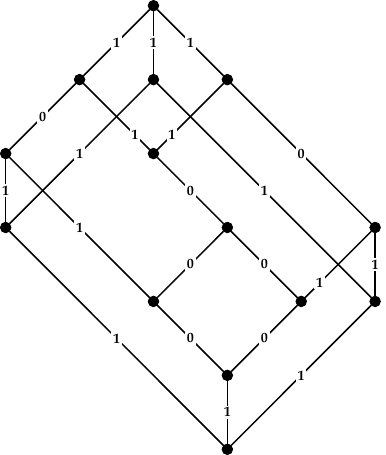}
	\caption{The atom-perspectivity marking of a non-Tamari Cambrian lattice of $A_3$.}
	\label{fig:cambrian_3}
\end{figure}

% \bib, bibdiv, biblist are defined by the amsrefs package.

\end{document}